\documentclass[11pt]{article}

\usepackage[left=1in,right=1in,top=1in,bottom=1in]{geometry}

\usepackage{amsthm}
\usepackage[pagebackref,colorlinks]{hyperref}
    \hypersetup{linkcolor=black, 
        citecolor=black}
\usepackage{amsmath}

\usepackage{mathrsfs}
\usepackage[nameinlink]{cleveref}
    
    \crefname{ex}{Example}{Examples}
    \crefname{thm}{Theorem}{Theorems} 
    \crefname{lem}{Lemma}{Lemmas}
    \crefname{prop}{Proposition}{Propositions}
    \crefname{cor}{Corollary}{Corollaries} 
    \crefname{conj}{Conjecture}{Conjectures} 
    \crefname{defn}{Definition}{Definitions}
    \crefname{rmk}{Remark}{Remarks} 
    \crefname{equation}{Eq.}{Eqs.} 
    \crefname{figure}{Fig.}{Figs.}
    
	\newtheorem{thm}{Theorem}[section]
	\newtheorem{lem}[thm]{Lemma}

	\newtheorem*{thm*}{Theorem}
	\newtheorem*{cor*}{Corollary}
	\theoremstyle{definition} 
		\newtheorem{defn}[thm]{Definition}
		\newtheorem{ex}[thm]{Example}
    	\newtheorem{rmk}[thm]{Remark}	
    \newcommand{\df}[1]{{\bf\emph{{#1}}}}

\usepackage{tikz, graphicx, pgfplots}
    \pgfplotsset{compat=1.18}
\usepackage{tikz-cd}
\usepackage[margin=0.5cm]{caption}
\usepackage{subcaption}
	\usetikzlibrary{positioning}
	\usetikzlibrary{decorations.pathmorphing}
    \usetikzlibrary{arrows}
    \usepackage{nicefrac}
    
        \tikzset{%
        fwdrxn/.style={very thick, arrows={-latex'}, blue},
        revrxn/.style={very thick, arrows={-Stealth[length=5pt,width=5pt,left]}},
        newt/.style={turq, opacity=0.15},
        vertex/.style={shape=circle, fill=blue, minimum size=5pt, inner sep=0, outer sep=3pt}
        }
\usepackage{color, xcolor}
    
    \newcommand\blue[1]{{\textcolor{blue}{#1}}}
    
    \definecolor{goldenyellow}{RGB}{253,231,36}
    \definecolor{ratecnst}{RGB}{172,172,172}
	
\usepackage{multirow}
\usepackage{array}

\usepackage{parskip}
\usepackage{amsfonts, amssymb, amsrefs, mathtools}
\newcommand{\eq}[1]{\begin{align*}#1\end{align*}}
	\newcommand{\eqn}[1]{\begin{align}#1\end{align}}  
	
\newcommand\mbf[1]{\mathbf{#1}}

\newcommand{\rr}{\ensuremath{\mathbb{R}}}

\renewcommand{\epsilon}{\varepsilon}	
\renewcommand{\phi}{\varphi}			

\DeclareMathOperator{\diag}{\mathbf{diag}}		
\DeclareMathOperator{\Span}{span}		
\newcommand{\braket}[2]{\left\langle{#1},\,{#2}\right\rangle}	
\newcommand{\kk}{\kappa}

\newcommand{\vv}[1]{{\boldsymbol{#1}}}  
\newcommand{\mm}[1]{\mathbf{#1}}               

\newcommand{\rrpp}{\rr_{>0}}

\newcommand{\Ak}{\mm A_{\vv\kk}}
\newcommand{\xx}{\vv x}
\newcommand{\xxi}{\vv \xi}
\newcommand{\yy}{\vv y}

\newcommand{\ehat}{\hat{\mbf e}}
\newcommand{\vc}[1]{\ensuremath{\begin{pmatrix}#1\end{pmatrix}}}
\usepackage{bbm}  

\usepackage{enumitem}
\usepackage[version=4]{mhchem}
\usepackage{chemfig}

\newcommand{\GLV}{GLV}

\usepackage{comment}
\usepackage[colorinlistoftodos,prependcaption,textsize=scriptsize,  
    linecolor=red!40!orange, bordercolor=red!40!orange, backgroundcolor=goldenyellow, 
    textcolor=black]{todonotes}

\usepackage{authblk}\usepackage[symbol]{footmisc}
\title{
   Generalized Lotka--Volterra Systems and Complex Balanced Polyexponential Systems
}
\author[1]{
        Diego Rojas La Luz%
}\author[2]{
        Polly Y. Yu%
}\author[3]{
        Gheorghe Craciun%
}
\affil[1]{\small Department of Mathematics, University of Wisconsin--Madison}
\affil[2]{\small Department of Mathematics, University of Illinois Urbana-Champaign}
\affil[3]{\small Department of Mathematics and Department of  Biomolecular Chemistry, University of Wisconsin--Madison}

\date{} 

\begin{document}

\maketitle
\renewcommand*{\thefootnote}{\arabic{footnote}}

\begin{abstract}
    \noindent
    We study the global stability of generalized Lotka--Volterra systems with generalized polynomial right-hand side,  without restrictions on the number of variables or the polynomial degree, including negative and non-integer degree. We introduce {\em polyexponential dynamical systems}, which are equivalent to the generalized Lotka--Volterra systems, and we use an analogy to the theory of mass-action kinetics to define and analyze complex balanced polyexponential systems, and implicitly analyze {\em complex balanced generalized Lotka--Volterra  systems}. We prove that complex balanced generalized Lotka--Volterra  systems have globally attracting states, up to standard conservation relations, which become linear for the associated polyexponential systems. In particular, complex balanced generalized Lotka--Volterra systems {\em cannot} give rise to periodic solutions, chaotic dynamics, or other  complex dynamical behaviors. We describe a simple sufficient condition for complex balance in terms of an associated graph structure, and we use it to analyze specific examples.
\end{abstract}

\section{Introduction}
\label{sec:introduction}

In the study of ecological systems, very extensive work has been dedicated to the analysis of dynamical systems defined on $\rrpp^n$ of the form
\eqn{\label{eq:GLV_intro}
\dfrac{dx_i}{dt}=x_i f_i(x_1,\dots,x_n),\quad \text{for }i=1,\dots,n,
}
in which $x_i$ represents the population of the $i$th species and $f_i(x_1,\dots,x_n)$ represents the per capita growth rate of that species. 
Lotka \cite{Lotka2002} and Volterra \cite{Volterra1926} were the first to study these types of systems, and Kolmogorov later extended their scope \cite{Kolmogorov1936}. Here we refer to such models as {\em generalized Lotka--Volterra systems}. 
The dynamics of these types of model can be extremely varied: Smale \cite{smale1976} and Hirsch \cite{hirsch1988-III} showed that competitive Lotka--Volterra systems can in general exhibit any asymptotic behavior, i.e., they can give rise to fixed points, limit cycles, $n$-torus attractors or even chaotic dynamics. However, despite this, there has been extensive work on finding sufficient conditions for global stability of these systems~\cites{Harrison1979,Hsu1978,Smith1986,Goh1977,Hofbauer_Sigmund_1998,Goh2012,Smith2011,Zhao2003}. 
Our goal in this paper is to provide a new framework for finding sufficient conditions for global stability of generalized Lotka--Volterra systems, inspired by corresponding results for {\em mass-action systems}. In particular, with this new framework we are able to recover some of the classical theory of global stability for Lotka--Volterra systems, and also to provide new results on global stability, which we showcase through examples that contain higher-order interaction terms.

In many cases, generalized Lotka--Volterra systems can be interpreted as mass-action kinetics, but almost always fall outside the theoretical framework developed by mass-action kinetics pioneers such as Horn, Jackson, and Feinberg~\cites{HornJackson1972,FeinbergHorn1977,Feinberg2019}. While the classical theory introduced by Horn--Jackson provides powerful tools for analyzing mass-action systems, its applicability to Lotka-–Volterra models is limited.

To illustrate this limitation, consider the dynamical system
    \eqn{
    \begin{split}\label{eq:intro-ex}
        \dfrac{dx_1}{dt} &= x_1(I_1 x_1^{-2} + r_1 - a_{11}x_1 + a_{12}x_2 - b_{11}x_1^2)\\
        \dfrac{dx_2}{dt} &= x_2(I_2 x_2^{-1} + r_2 + a_{21}x_1 - a_{22}x_2 - b_{21}x_1x_2 - b_{22}x_2^{3/2}),
    \end{split}
    }
which is a cooperative generalized Lotka--Volterra system with immigration terms and higher-order interactions. The global stability of this system cannot be deduced from classical results for mass-action systems. For example, there is no weakly reversible realization~\cite{Craciun2020} of this system, and there is no obvious way to adapt the classical theory to analyze it.

Mathematical analysis of mass-action systems often relies on the graph-theoretical structure of {\em reaction networks}, which are graphs associated to the differential equations, in that some structural properties of this graph imply certain dynamical properties, such as global stability. However, this mass-action graph is not well suited for Lotka--Volterra models.
{\em Instead, we propose a different way to associate a graph to a generalized Lotka--Volterra system, such that structural properties of \underline{this} graph imply dynamical properties  of the generalized Lotka--Volterra system.} 

The  framework we introduce here  extends the applicability of classical results, and also unveils new insights into the global stability and behavior of these systems.

The main result of this paper is an extension of the Horn--Jackson theorem, which states that for mass-action kinetic systems, a complex balanced steady state is {\em locally stable}, as shown by a Lyapunov function. We extend this result to the setting of generalized Lotka--Volterra systems, and show that if a steady state is complex balanced -- defined in a manner that extends the classical definition -- then it is {\em globally stable}. This result is significant because, while global stability of complex balanced steady states remains {\em a long-standing conjecture}\footnote{This is the well-known {\em global attractor conjecture}, which is  one of the main  open problems in the theory of mass-action systems~\cite{Feinberg2019}.} for mass-action systems, we prove it here for generalized Lotka--Volterra systems. Specifically, we show that:

\begin{thm}
\label{thm:intro}
    Consider a generalized Lotka--Volterra system on $\rrpp^n$ that admits at least one  complex balanced steady state.
    Then every positive steady state is complex balanced, and there exists a foliation of the state space $\rrpp^n$  into compatibility manifolds, such that each compatibility manifold contains exactly one complex balanced steady state $\xx^*$, and this state is globally attracting within its compatibility manifold. In particular, the function
    \eq{
        V(\xx) = \sum_{i=1}^n (\log x_i - \log x_i^*)^2,
    }
    is a global strict Lyapunov function within the compatibility manifold of $\xx^*$, where it has a strict   global minimum at $\xx = \xx^*$.
\end{thm}

In other words, in order to conclude global stability, it suffices to check a condition called {\em complex balance}, which consists of some algebraic identities. Moreover, we  also describe sufficient conditions for complex balance that are very easy to check; these conditions are essentially linear, and have to do with geometric properties of a specific {\em graph embedding}  that generates our generalized Lotka--Volterra system. 

\subsection{Structure of paper}

In \Cref{sec:GLV}, we describe how generalized Lotka--Volterra systems are associated to {\em Euclidean embedded graphs}, which serve as a bridge to the theory of {\em mass-action systems}, and inspiration for some key mathematical approaches. 
In \Cref{sec:polyexp}, we introduce our other main mathematical tool, i.e., {\em polyexponential systems}, and we show that they also can be generated by Euclidean embedded graphs. Moreover, we show that polyexponential systems are  equivalent to generalized Lotka--Volterra systems, via a simple change of variables. Once we have developed this machinery, we prove a theorem about polyexponential systems (\cref{thm:polyexp-HJthm}), which provides a sufficient condition for global stability. 
Then, in \Cref{sec:CB-GLV} we return to generalized Lotka--Volterra systems, and reformulate the results in the previous section to obtain our main result (\cref{thm:intro,thm:glv-cb}). We also prove an extension (\cref{thm:diag-cb-glv}) that is useful for applications. We conclude with a section on examples and one that discusses avenues for future work.

\subsection{Definitions and notation}
\label{sec:def-notations}

Throughout this work, we let $\rrpp$ denote the set of positive real numbers. Accordingly, $\rrpp^n$ denotes the set of vectors in $\rr^n$ with positive entries. Bold symbols denote vectors, e.g., $\xx = (x_1,\dots,x_n)^\top$. We write $\vv x > \vv 0$ to mean that all components of $\vv x$ are positive. 
For any $\xx \in \rrpp^n$ and $\yy \in \rr^n$, let $\xx^\yy = x_1^{y_1} x_2^{y_2} \cdots x_n^{y_n}$.
If $\mm Y \in \rr^{n \times m}$ is a matrix with column vectors $\yy_1, \ldots, \yy_m$, then $\xx^{\mm Y} = (\xx^{\yy_1},\ldots, \xx^{\yy_m})^\top$. 
The standard inner product of $\vv v, \vv w \in \rr^n$ is denoted $\braket{\vv v}{ \vv w}$. 
We extend $\exp$ and $\log$ componentwise to vectors in $\rr^n$ and $\rrpp^n$ respectively. By a straightforward calculation,  $\log (\xx^{\mm Y}) = \mm Y^\top \log \xx$. Finally, denote by $\ehat_1, \ehat_2,\dots, \ehat_n$ the standard basis of $\rr^n$, and $\diag(\xx)$ the diagonal matrix with entries $x_1,\ldots, x_n$. 

\section{Generalized Lotka--Volterra systems}
\label{sec:GLV}

Consider the dynamical system on $\rrpp^n$ given by
\begin{equation}
\label{eq:LV-general}
\dfrac{dx_i}{dt} = x_i f_i(x_1,\dots,x_n), \qquad \text{for } i =1,\ldots, n, 
\end{equation}
where the functions $f_i(\xx)$ are {\em generalized polynomials}, i.e., linear combinations of monomials with real (instead of integer) exponents; see \Cref{eq:intro-ex} for an example. We refer to systems of the form \eqref{eq:LV-general} as \df{generalized Lotka--Volterra systems}, or simply \df{GLV systems}. 
These classes of dynamical systems are widely studied as models of population dynamics, infectious disease,  evolutionary game theory, and many other applications~\cites{Hofbauer_Sigmund_1998,Goh2012,Smith2011,Zhao2003}.

In this paper we will identify a large class of \GLV\ systems (called {\em complex balanced \GLV\ systems}) that have {\em remarkably stable dynamics:} they do not allow oscillations or chaotic dynamics, they cannot have solutions that converge to infinity or to the boundary of the positive orthant, and actually, {\em any solution that starts in the positive orthant converges to a positive steady state}; moreover, each steady state is unique within a invariant manifold, which can be parametrized explicitly.

Many of the methods in our analysis are strongly inspired by the theory of {\em mass-action systems}~\cite{Yu2018}. A key step is to represent a dynamical system of interest using an embedded graph~\cite{craciun2019polynomial}. 

\begin{defn}\label{def:Egraph}
    An \df{Euclidean embedded graph} (or \df{E-graph}) is a finite directed graph $G = (V,E)$ embedded in $\rr^n$, with no self-loops, i.e., with $V \subset \rr^n$ and $E \subset V \times V$ and $(\yy,\yy) \notin E$ for any $\yy \in V$. 
\end{defn}

Let $n$ denote the dimension of the state space (i.e., the number of species), and $m$ denote the number of vertices in the graph, i.e., $m = |V|$. We enumerate the vertices, i.e., $V = \{\vv y_1,\ldots,\vv y_m\}$. 
An edge $(\yy_i, \yy_j) \in E$ is also denoted $\yy_i \to \yy_j$, where the point $\vv y_i$ is said to be its \df{source vertex}, and $\yy_j$ is its \df{product vertex}. We may also write $(i, j) \in E$ instead of $(\yy_i, \yy_j) \in E$. 
We say that an E-graph is \df{reversible} if $(i,j) \in E$ whenever $(j,i) \in E$; it is \df{weakly reversible} if every edge is part of a directed cycle.

In the theory of mass-action systems, a dynamical system of interest can be generated  from a reaction network (interpreted as an E-graph $G$) along with a choice of positive edge weights $\vv\kk \in \rrpp^{E}$.  
The goal of {\em reaction network theory} is to deduce parameter-independent dynamical information from the graph-theoretic and geometric properties of $G$. {\em In this work, we aim for a similar theory for \GLV\  systems.} We first specify how a \GLV\ system can be generated from $(G,\vv\kk)$.

\begin{defn}\label{def:glv-from-graph}
    Let $G = (V,E)$ be an E-graph and $\vv\kk = (\kk_{ij})_{(i,j)\in E} \in \rrpp^{E}$ be a choice of edge weights, and define 
       $$
       \vv f_{(G,\vv\kk)}(\xx)
       = \sum_{(i,j)\in E} \kk_{ij} \xx^{\yy_i} (\yy_j - \yy_i).
       $$ 
    The \df{\GLV\ system generated by} $(G, \vv\kk)$ is the system of differential equations on $\rrpp^n$
    \eqn{\label{eq:GLV}
        \dfrac{d\xx}{dt} = \diag(\xx) \vv f_{(G,\vv\kk)}(\xx)
        . 
    }
\end{defn}

\begin{rmk}\label{rmk:MAK}
The function $\vv f_{(G,\vv\kk)}(\xx)$
gives exactly the general formula of the right-hand side of a {\em mass-action system}. In other words, in the theory of mass-action systems, the differential equations
\eqn{
\label{eq:MAK}
   \dfrac{d\xx}{dt} = \vv f_{(G,\vv\kk)}(\xx)
}
represents the  {\em mass-action system generated by $(G,\vv\kk)$}~\cite{Feinberg2019,Yu2018}. In general, the dynamical properties of the systems \eqref{eq:GLV} and \eqref{eq:MAK} may be different; nevertheless, we will take advantage of this connection. 
\end{rmk}

\begin{ex}\label{ex:2d-stoich-glv}
    Consider the E-graph $G$ with vertex set $V = \{\ehat_1, \ehat_2, \ehat_3\}$, and edge set $E = \{ {\ehat_1 \rightleftharpoons \ehat_2}, {\ehat_2 \rightleftharpoons \ehat_3}, {\ehat_3 \rightleftharpoons \ehat_1} \}$, as shown in \Cref{fig:2d-stoich-a}. Choosing edge weights $\vv\kk = (\kk_{ij})_{(i,j)\in E} > \vv 0$ generates the \GLV\ system
    \eqn{\label{eq:2d-stoich-glv}
    \begin{split}
        \dfrac{dx_1}{dt} &= x_1 \left((-\kk_{12}-\kk_{13})x_1 +\kk_{21}x_2 + \kk_{31}x_3 \right)\\
        \dfrac{dx_2}{dt} &= x_2 \left(\kk_{12}x_1 + (-\kk_{21} -\kk_{23})x_2 + \kk_{32}x_3 \right)\\
        \dfrac{dx_3}{dt} &= x_3 \left(\kk_{13}x_1 + \kk_{23}x_2 + (-\kk_{31}-\kk_{32})x_3 \right),
    \end{split}
    }
    whose vector field (with all $\kk_{ij} = 1$) is shown on two invariant surfaces in \Cref{fig:2d-stoich-b}.
\end{ex}

    \begin{figure}[h!t]
    \centering
    \begin{subfigure}[b]{0.25\textwidth}
    \centering 
        \begin{tikzpicture}[scale=1]
        \begin{scope}[shift={(5,-3)}]
        \begin{axis}[
          view={115}{25},  
          axis lines=center,
          axis line style={gray, thick},
          width=5.5cm,height=5cm,
          ticks = none, 
          xmin=0,xmax=1.7,ymin=0,ymax=1.6,zmin=0,zmax=1.5,
        ]
        \addplot3 [no marks, dashed, gray!70] coordinates {(1,0,0)  (1,0,2.5)};
        \addplot3 [no marks, dashed, gray!70] coordinates {(0,1,0)  (0,1,2.5)};
        \addplot3 [no marks, dashed, gray!70] coordinates {(2.5,0,1)  (0,0,1) (0,2.5,1)};
        \addplot3 [no marks, dashed, gray!70] coordinates {(1,0,0)  (1,2.5,0) };
        \addplot3 [no marks, dashed, gray!70] coordinates {(0,1,0)  (2.5,1,0) };
        
        \addplot3 [only marks, blue] coordinates {(0,0,1) (0,1,0) (1,0,0) };
        
        \node [outer sep=1pt] (3) at (axis cs:0,0,1) {};
        \node [outer sep=1pt] (2) at (axis cs:0,1,0) {};
        \node [outer sep=1pt] (1) at (axis cs:1,0,0) {};
        
        \draw [fwdrxn, transform canvas={xshift=0.3ex,yshift=0.25ex}] (3)--(2); 
        \draw [fwdrxn, transform canvas={xshift=-0.3ex,yshift=-0.25ex}] (2)--(3); 
        \draw [fwdrxn, transform canvas={xshift=-0.05ex,yshift=0.35ex}] (1)--(2); 
        \draw [fwdrxn, transform canvas={xshift=0.05ex,yshift=-0.35ex}] (2)--(1); 
        \draw [fwdrxn, transform canvas={xshift=0.3ex,yshift=-0.25ex}] (3)--(1); 
        \draw [fwdrxn, transform canvas={xshift=-0.3ex,yshift=0.25ex}] (1)--(3); 
    
        \end{axis} 
        \end{scope}
        \end{tikzpicture}
    \caption{}
    \label{fig:2d-stoich-a}
    \end{subfigure}
    \hspace{0.15cm}
    \begin{subfigure}[b]{0.33\textwidth}
    \centering 
        \includegraphics[height=1.6in]{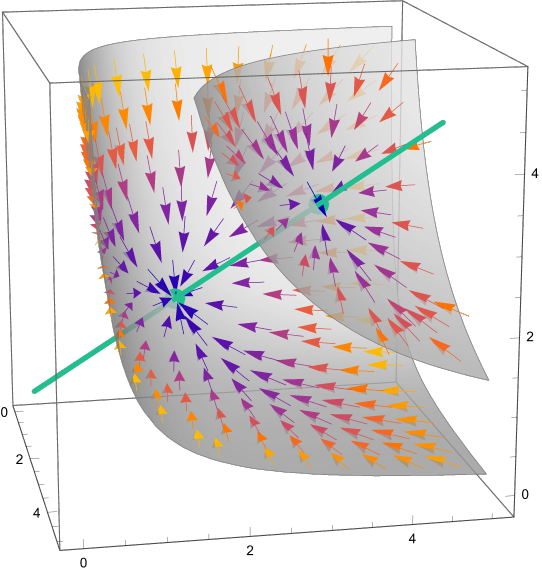}
    \caption{}
    \label{fig:2d-stoich-b}
    \end{subfigure}
    \hspace{0.5cm}
    \begin{subfigure}[b]{0.3\textwidth}
    \centering 
        \includegraphics[height=1.6in]{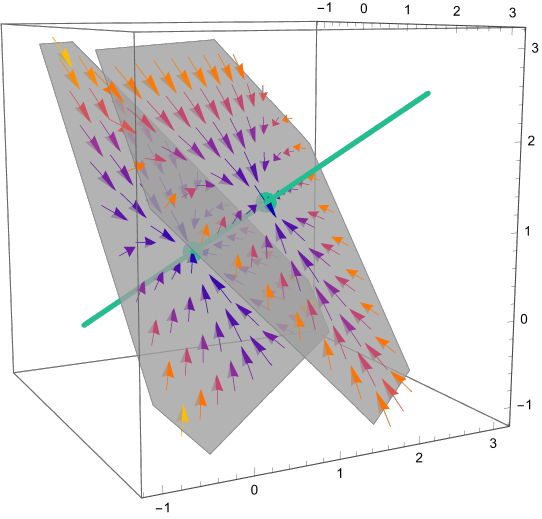}
    \caption{}
    \label{fig:2d-stoich-c}
    \end{subfigure}
    \caption{(a) The E-graph $G$ of \Cref{ex:2d-stoich-glv}, embedded in $\rr^3$. (b-c) Vector fields on invariant sets for (b) the \GLV\ system \eqref{eq:2d-stoich-glv}, and (c) after the change of variables $x_i \mapsto \log(x_i)$. The green line represents the steady state set. Note that in both (b) and (c) there is a foliation of the state space, given by invariant sets.}
    \label{fig:2d-stoich}
    \end{figure}
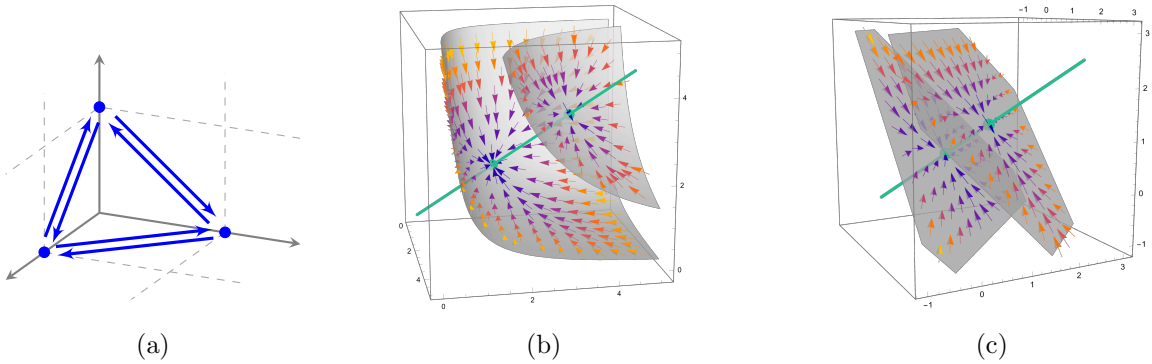 

As we mentioned, we can take advantage of the similarities between the right-hand sides of \eqref{eq:GLV} and \eqref{eq:MAK}. 
In particular, the vector-valued function $\vv f_{(G,\vv\kk)}(\xx)$ admits a matrix-form representation based on the structure of $(G,\vv\kk)$. 
Let $\mm Y \in \rr^{n \times m}$ be the matrix whose columns are the vertices $\yy_1, \ldots, \yy_m$. Let $\kk_{ij}$ be the edge weight of $\yy_i \to \yy_j$. The \df{Kirchoff matrix} $\Ak \in \rr^{m\times m}$ is the negative transpose of the Laplacian matrix of the weighted directed graph $(G,\vv\kk)$, i.e., 
\eq{
[\Ak]_{ij} = \begin{cases}
    \kk_{ji}, &\quad \text{ if }(i,j)\in E,\\
    \displaystyle-\sum_{\ell \colon \!   (i,\ell)\in E} \!\!\! \kk_{i\ell}, &\quad \text{ if } i=j,\\
    0, &\quad \text{otherwise}. 
\end{cases}
}
Then $\vv f_{(G,\vv\kk)}(\xx) = \mm Y \Ak \xx^{\mm Y}$~\cite{Feinberg2019}, and it follows that \eqref{eq:GLV} can be written as $\dfrac{d\xx}{dt} = \diag(\xx) \mm Y \Ak \xx^{\mm Y}$.

The E-graph $G$ is not only a directed graph, but because of its embedding, it contains geometric information. Again borrowing terms from reaction network theory, we define the notions of the {\em stoichiometric subspace} and {\em deficiency} \cite{Feinberg2019,Yu2018}, the former being the span of all vectors associated to the edges, while the latter measures how far is the embedding of $G$ from being generic (see \cite{Craciun2021} for a more in-depth discussion).

\begin{defn}\label{def:stoichiometric_subspace}
    The \df{stoichiometric subspace} of an E-graph $G$ is the vector space $S \coloneqq  \text{span}_{\rr} \{ \yy_i - \yy_j \colon  (i,j) \in E \}$. 
\end{defn}

\begin{defn}\label{def:deficiency}
    The \df{deficiency} of an E-graph $G$ is $\delta \coloneqq  |V| - \ell - \dim S$, where $\ell$ is the number of connected components of $G$, and $S$ is the stoichiometric subspace.
\end{defn}

\begin{ex}\label{ex:2d-stoich-polyexp}
    Considering the E-graph $G$ from \cref{ex:2d-stoich-glv} or \Cref{fig:2d-stoich-a}. The stoichiometric subspace of $G$ is the 2-dimensional set $S = \Span_{\rr} \left\{ {(-1,1,0)^\top,}\ {(0,-1,1)^\top,}\ {(1,0,-1)^\top}\right\}$. The deficiency of $G$ is $\delta = 3 - 1 - 2 = 0$.

    Consider the change of variables $\xxi = \log(\xx)$. Since $\dfrac{d\xi_i}{dt} = \dfrac{1}{x_i} \dfrac{dx_i}{dt}$, this results in the following system defined on $\rr^3$: 
    \eqn{\label{eq:2d-stoich-polyexp}
    \begin{split}
        \dfrac{d\xi_1}{dt} &= (-\kk_{12}-\kk_{13})e^{\xi_1} +\kk_{21}e^{\xi_2} + \kk_{31}e^{\xi_3}\\
        \dfrac{d\xi_2}{dt} &= \kk_{12}e^{\xi_1} + (-\kk_{21} -\kk_{23})e^{\xi_2} + \kk_{32}e^{\xi_3}\\
        \dfrac{d\xi_3}{dt} &= \kk_{13}e^{\xi_1} + \kk_{23}e^{\xi_2} + (-\kk_{31}-\kk_{32})e^{\xi_3}. 
    \end{split}
    }
    Because $\dfrac{d\xxi}{dt} = \sum_{(i,j)\in E} \kk_{ij} e^{\braket{\vv\xi}{ \yy_i }} (\yy_j - \yy_i) \in S$, the set $S_{\xxi_0} \coloneqq S + \xxi_0$ is \emph{invariant} for \eqref{eq:2d-stoich-polyexp} for any $\xxi_0\in \rr^3$, i.e., if $\xxi(0) \in S_{\xxi_0}$, then $\xxi(t) \in S_{\xxi_0}$ for all $t \in \rr$ for which the solution is defined. \Cref{fig:2d-stoich}(c) shows the vector field on two such invariant sets. 
\end{ex}

What we observed in \Cref{ex:2d-stoich-polyexp} holds more generally. We call systems like \eqref{eq:2d-stoich-polyexp} {\em polyexponential systems}, and they will be the focus of the next section.

\section{Polyexponential dynamical systems}
\label{sec:polyexp}

We have seen that the change of variables $\xi_i = \log(x_i)$ converts a \GLV\ system $\dfrac{d\xx}{dt} = \diag(\xx) \vv f_{(G,\vv\kk)}(\xx)$ into a system of differential equations in $e^{\xi_i}$ variables:  $\dfrac{d\xxi}{dt} = \vv f_{(G,\vv\kk)}(e^\xxi)$. We show in this section that this latter system is in some sense easier to analyze, despite the two being equivalent.

\begin{defn}\label{def:polyexponential}
    Let $G = (V,E)$ be an E-graph and $\vv\kk = (\kk_{ij})_{(i,j)\in E} \in \rrpp^{E}$ be a choice of edge weights, and define
    \eq{
    \vv \phi_{(G,\kk)}(\xxi) = \sum_{(i,j)\in E} \kk_{ij} e^{\braket{\vv\xi}{ \yy_i }} (\yy_j - \yy_i).
    }
    The \df{polyexponential system generated by} $(G,\vv\kk)$ is the system of differential equations on $\rr^n$   
    \eqn{\label{eq:polyexponential}
    \dfrac{d\vv\xi}{dt} =  \vv \phi_{(G,\kk)}(\xxi), 
    }
    or in matrix-form, $\dfrac{d\xxi}{dt} = \mm Y \Ak e^{\mm Y^\top \xxi}$.
\end{defn}

\begin{rmk}
    The dynamical system \eqref{eq:GLV},  defined on $\rrpp^n$, is equivalent to the dynamical system \eqref{eq:polyexponential}, defined on $\rr^n$, via the change of variables $x_i \mapsto \xi_i = \log(x_i)$. For example, the system \eqref{eq:2d-stoich-glv} is equivalent to the system \eqref{eq:2d-stoich-polyexp}; see \Cref{fig:2d-stoich-b,fig:2d-stoich-c}. 
\end{rmk}

As we observed in \cref{ex:2d-stoich-polyexp}, translates of the stoichiometric subspace are invariant under \eqref{eq:polyexponential}.

\begin{lem}\label{lem:compatibility_class}
    Consider the polyexponential system \eqref{eq:polyexponential} generated by $(G,\vv\kk)$, with stoichiometric subspace $S$. Then for any $\xxi_0 \in \rr^n$, its \df{compatibility class} $S_{\xxi_0} \coloneqq S + \xxi_0$ is invariant.
\end{lem}
\begin{proof}
   From \eqref{eq:polyexponential}, we see that the right-hand side of $\dfrac{d\xxi}{dt}$ is a linear combination of the  vectors of the form $\yy_j - \yy_i$. Thus, if a function $\xxi(t)$ is a solution of \eqref{eq:polyexponential} on an interval that contains $t=0$, then $\xxi(t) - \xxi(0) \in S$ for all $t$ for which $\xxi(t)$ is defined.
\end{proof}

\subsection{Complex balanced polyexponential systems}

Inspired by reaction network theory, we analyze global stability of polyexponential systems (thus implicitly \GLV\ systems) that are generated by E-graphs with specific properties. 
Recall that a system of differential equations $\dfrac{d\xx}{dt} = \vv g(\xx)$  is said to be \df{globally stable} on a forward-invariant set $\Omega$ if there exists a steady state $\xx^* \in \Omega$ such that for any initial condition $\xx(0) \in \Omega$, the solution $\xx(t)$ converges to $\xx^*$.

The key concept of a {\em complex balanced} system  was originally defined in the context of mass-action systems~\cite{HornJackson1972}; here we  reinterpret it for polyexponential systems. This reinterpretation will allow us to extend some powerful stability results in  reaction network theory~\cite{Yu2018} to a broad class of polyexponential systems.

\begin{defn}
The polyexponential system \eqref{eq:polyexponential} generated by $(G,\vv\kk)$ is said to be \df{complex balanced} if there exists $\vv\xi^* \in \rr^n$ such that for any $\yy_i \in V$, we have
    \begin{equation}\label{eq:cb}
        \sum_{j \colon (i,j)\in E} \kk_{ij} e^{\braket{\vv\xi^*}{\yy_i}} 
        = \sum_{j \colon (j,i)\in E} \kk_{ji} e^{\braket{ \vv\xi^*}{ \yy_j}}.
    \end{equation}
    The point $\xxi^*$ is called a \df{complex balanced steady state}.
\end{defn}

\begin{rmk}\label{rmk:CB-polyexp-ker_Ak}
    \Cref{eq:cb} is equivalent to $e^{\mm Y^\top \xxi^*} \in \ker \Ak$~\cite{Feinberg2019}, 
    from which it follows that $\xxi^*$ is a steady state of the polyexponential system \eqref{eq:polyexponential}. One way to interpret \eqref{eq:cb} is that at every $\yy_i \in V$, the total inflow into $\yy_i$ equals the total outflow from $\yy_i$, when evaluated at $\xxi^*$.
    (Informally, this is analogous to {\em Kirchoff's current law}, which says that the total current that enters a node of an electrical circuit equals the total current that exits that node.)
    The name \lq\lq complex balanced\rq\rq\ comes from the theory of mass-action systems, where vertices are called \lq\lq complexes\rq\rq~\cite{HornJackson1972}.
\end{rmk}

\begin{ex}\label{ex:2d-stoich-polyexp-cb}
    Consider the polyexponential system given by \eqref{eq:2d-stoich-polyexp} in \Cref{ex:2d-stoich-polyexp}. For  $\xxi^* \in \rr^3$ to be complex balanced, it needs to satisfy the following equations:
    \eqn{\label{eq:2d-stoich-polyexp-cb}
    \begin{split}
    \kk_{12}e^{\xi_1^*} + \kk_{13}e^{\xi_1^*} &= \kk_{21}e^{\xi_2^*} + \kk_{31}e^{\xi_3^*}\\
    \kk_{21}e^{\xi_2^*} + \kk_{23}e^{\xi_2^*} &= \kk_{12}e^{\xi_1^*} + \kk_{32}e^{\xi_3^*}\\
    \kk_{31}e^{\xi_3^*} + \kk_{32}e^{\xi_3^*} &= \kk_{13}e^{\xi_1^*} + \kk_{23}e^{\xi_2^*}.
    \end{split}
    }
    By \cref{rmk:CB-polyexp-ker_Ak}, we may assume $\xxi^*$ is a steady state of \eqref{eq:2d-stoich-polyexp}. For this example, any steady state will satisfy \eqref{eq:2d-stoich-polyexp-cb}, which we prove is a consequence of deficiency of $G$ being zero. 
\end{ex}

\subsection{Global stability of complex balanced polyexponential systems}

In this section, we analyze polyexponential systems that are complex balanced. We start by showing the  dynamical implications of the complex balance condition, i.e., the condition that there exists at least one complex balanced steady state. Then, we characterize when a system is complex balanced {\em for any} choice of edge weights.

We first prove that the existence of a complex balanced steady state leads to a Lyapunov function.

\begin{lem}\label{lem:polyexp-Lyapunov}\label{thm:global_stab}
    Consider the polyexponential system \eqref{eq:polyexponential}  generated by $(G,\vv\kk)$, with stoichiometric subspace $S$. Suppose there exists a complex balanced steady state  $\xxi^* \in \rr^n$. Define on $\rr^n$ the function
    \eqn{\label{eq:polyexp-Lyapunov}
        L(\xxi) = \sum_{i=1}^n (\xi_i - \xi^*_i)^2.
    }
    Then $L(\xxi)$ is a Lyapunov function for the system \eqref{eq:polyexponential}, i.e., we have $\vv \phi_{(G,\kk)}(\xxi) \cdot \nabla L(\xxi) \leq 0$ for all $\xxi \in \rr^n$. 
    Moreover, equality holds if and only if $\xxi - \xxi^* \in S^\perp$.
\end{lem}
\begin{proof}
    Note that by the mean value theorem $e^z (z' - z) \leq e^{z'} - e^{z}$ for any $z,z' \in \rr$. Since the gradient is $\nabla L(\xxi) = 2  (\xxi - \xxi^*)$,  we have
    \begin{align*}
        \vv \phi_{(G,\kk)}(\xxi) \cdot \nabla L(\xxi) &= 2 \sum_{(i,j)\in E} \kk_{ij} e^{\braket{\yy_i}{\xxi - \xxi^* + \xxi^*} }  \braket{ \yy_j - \yy_i }{ \xxi - \xxi^* } \\
        &= 2 \sum_{(i,j)\in E}  \kk_{ij} e^{ \braket{ \yy_i}{\xxi^*} } e^{ \braket{ \yy_i}{ \xxi - \xxi^*} }  \left[ \braket{\yy_j}{\xxi - \xxi^*} - \braket{\yy_i}{\xxi - \xxi^*} \right] \\
        & \leq 2 \sum_{(i,j)\in E}  \kk_{ij} e^{ \braket{ \yy_i}{\xxi^*} } \left[ e^{ \braket{\yy_j}{\xxi - \xxi^*} } - e^{ \braket{\yy_i}{ \xxi - \xxi^*} } \right]
        , 
    \end{align*}
    where equality holds if and only if $\braket{\yy_j - \yy_i}{\xxi - \xxi^*} = 0$ for all $\yy_i \to \yy_j \in E$, i.e., if and only if $\xxi-\xxi^* \in S^\perp$. Now, instead of summing over all edges, we sum over all vertices and edges coming into or out of each vertex; this rearrangement leads to 
    \begin{align*}
        \vv \phi_{(G,\kk)}(\xxi) \cdot \nabla L(\xxi) &\leq 2 \sum_{i=1}^m \left( 
                \sum_{j \colon (j,i) \in E} \kk_{ji} e^{\braket{\yy_j}{\xxi^*}} 
                - \sum_{j \colon (i,j) \in E} \kk_{ij} e^{\braket{\yy_i}{\xxi^*}}
            \right)  e^{\braket{\yy_i}{\xxi - \xxi^*}} = 0,
    \end{align*}
    as the differences in the parentheses are zero, due to $\xxi^*$ being complex balanced. 
\end{proof}

While the calculation above is inspired by one involving the Lyapunov function in \cite{HornJackson1972} for mass-action systems, an {\em important difference} between mass-action and polyexponential systems lie in the form of the Lyapunov function. Specifically, the  function  $L(\xxi)$ in \Cref{lem:polyexp-Lyapunov} is a \emph{proper Lyapunov function}, i.e., for any $C$, the set $\{ \xxi \colon L(\xxi) \leq C\}$ is a compact set within the state space of the polyexponential system, which is {\em not} true for the analogous function for mass-action system (see \cite{Craciun2009} for more details). One major advantage of a {\em proper} Lyapunov function is that it can be used to immediately conclude global stability (while, in the theory of mass-action kinetics, the global stability of complex balanced systems is still a conjecture, even though it has been formulated a long time ago~\cite{horn1974dynamics}).

\begin{thm}\label{thm:polyexp-HJthm}
    Consider the polyexponential system \eqref{eq:polyexponential}  generated from $(G,\vv\kk)$, with stoichiometric subspace $S$. Suppose $\xxi^*$ is complex balanced. Then the following hold.
    \begin{enumerate}
        \item\label{thm-polyexp-it:CB-affine_param} A state $\vv\zeta$ is a steady state if and only if   $\vv\zeta - \vv\xi^* \in  S^\perp$. 
        \item\label{thm-polyexp-it:allssisCB} Every steady state is complex balanced. 
        \item\label{thm-polyexp-it:unique_ss} There exists exactly one steady state within each compatibility class. 
        
        \item\label{thm-polyexp-it:Lyap_global_stab} The function \eqref{eq:polyexp-Lyapunov} is a global proper Lyapunov function for \eqref{eq:polyexponential}. When restricted to a compatibility class, the function $L(\xxi)$ has a unique minimum at the complex balanced steady state contained in that class.
        In particular, each steady state is globally attracting within its compatibility class. 
    \end{enumerate}
\end{thm}

\begin{proof}
    \begin{enumerate}
        \item 
        If $\vv\zeta$ is a steady state, we have $\vv \phi_{(G,\kk)}(\vv\zeta) \cdot \nabla L(\vv\zeta)=0$, so from \Cref{lem:polyexp-Lyapunov} we conclude that  $\vv\zeta - \vv\xi^* \in  S^\perp$.
        Conversely, if $\vv\zeta - \vv\xi^* \in  S^\perp$, then we have $\braket{\vv\zeta}{\yy_j-\yy_i} = \braket{\xxi^*}{\yy_j-\yy_i}$ for all $\yy_i \to \yy_j \in E$. This implies that, since $\xxi^*$ is complex balanced, for all $\yy_i \in V$ we have 
        \eq{
        \sum_{j \colon (i,j)\in E} \kk_{ij}
        = \sum_{j \colon (j,i)\in E} \kk_{ji} e^{\braket{ \vv\xi^*}{ \yy_j - \yy_i}} = \sum_{j \colon (j,i)\in E} \kk_{ji} e^{\braket{ \vv \zeta}{ \yy_j - \yy_i}},
        }
as the first equality is equivalent to \eqref{eq:cb}, and the second   equality follows from $\vv\zeta - \vv\xi^* \in  S^\perp$. Then the fact that the first and last term above are equal implies that $\vv\zeta$ is complex balanced, and therefore it is a steady state.

\item This follows from the proof of 1.

\item 
        Both existence and uniqueness follow immediately from the fact that any compatibility class is of the form $\xxi_0 + S$ while the set of steady states is of the form $\xxi^* + S^\perp$.  

        \item 
        This last claim follows from (i) the inequality in \Cref{lem:polyexp-Lyapunov}, (ii) that the sublevel sets of $L(\xxi)$ are balls, thus compact in $\rr^n$, and (iii) the unique minimum of $L(\xxi)$ in each compatibility class lies on $\xxi^* + S^\perp$. \qedhere
    \end{enumerate}
\end{proof}

\Cref{thm:polyexp-HJthm} is remarkable in that {\em global stability} of all steady states (within their compatibility classes) can be concluded just from the {\em existence of  one complex balanced steady state}. 
We now turn our attention to cases where the complex balance property can be established without having to check the equations~\eqref{eq:cb}.
Indeed, the deficiency of an E-graph  (\Cref{def:deficiency}) characterizes when the system is complex balanced for any choices of edge weights.

A complex balanced steady state $\xx^*$ for the mass-action system $\dfrac{d\xx}{dt} = \mm Y \Ak \xx^{\mm Y}$ can be defined as one satisfying $(\xx^*)^{\mm Y} \in \ker \Ak$ (see \cite{Yu2018, Feinberg2019}), while $\xxi^*$ is complex balanced for the polyexponential system $\dfrac{d\xxi}{dt} = \mm Y \Ak e^{\mm Y^\top \xxi}$ if and only if $e^{\mm Y^\top \xxi^*} \in \ker \Ak$. 
These relationships allow us to derive the theorems below by taking advantage of  analogous results from the theory of mass-action systems.

\begin{thm}\label{thm:WRD0isCBforallk} 
The polyexponential system generated by $(G,\vv\kk)$ is complex balanced for any $\vv\kk > \vv 0$ if and only if $G$ is weakly reversible and has deficiency zero.
\end{thm}
\begin{proof}
    The Deficiency Zero Theorem for mass-action systems states that $G$ being weakly reversible and  deficiency zero is necessary and sufficient for the system to have complex balanced steady state $\xx^*$ for any $\vv\kk  > \vv 0$~\cite{Yu2018, Feinberg2019}. Letting $\xxi^* = \log \xx^*$, we see that $(\xx^*)^{\mm Y} = (e^{T^\top \xxi^*})$. In other words, for any $\vv\kk > \vv 0$, the polyexponential system \eqref{eq:polyexponential} generated by $(G,\vv\kk)$ is complex balanced. 
\end{proof}

\begin{thm}\label{thm:DefZero}
    Suppose the E-graph $G$ has deficiency zero, and consider some arbitrary $\vv\kk > \vv 0$. 
    \begin{enumerate}
        \item If $G$ is weakly reversible, then within each compatibility class the polyexponential system generated by $(G,\vv\kk)$ has  a unique steady state that is complex balanced and globally stable. 
    
        \item If $G$ is not weakly reversible, then the polyexponential system generated by $(G,\vv\kk)$ cannot have any steady state, nor can it admit any periodic orbit.
    \end{enumerate}
    In particular, if the polyexponential system generated by $(G,\vv\kk)$  has a steady state, then this steady state is complex balanced and globally stable within its compatibility class. 
\end{thm}

\begin{proof}
    When $G$ is weakly reversible the result follows from \Cref{thm:WRD0isCBforallk,thm:polyexp-HJthm}.

    Now suppose $G$ is not weakly reversible. Then, by \cite[Proposition 16.5.3]{Feinberg2019}, $\sum_{(i,j)\in E} \alpha_{ij} (\yy_j-\yy_i) \neq \vv 0$ for all vectors $\vv \alpha = (\alpha_{ij})_{(i,j)\in E}> \vv 0$. By Stiemke's Lemma it follows that there exists a vector $\vv p$ such that the dot products $\braket{\vv p}{ \yy_j-\yy_i}$ are non-negative for all $(i,j)\in E$, and at least one such dot product is strictly positive. This implies that the function
    $V_{\vv p}(x) = -\sum_{i=1}^n p_ix_i$ 
    is a linear strict Lyapunov function for the polyexponential system \eqref{eq:polyexponential}  generated by $(G,\vv\kk)$, and therefore this system cannot have any steady state or periodic orbits.
\end{proof}

\begin{ex}\label{ex:2d-stoich-polyexp-Def0}
    In \cref{ex:2d-stoich-polyexp-cb}, we noted that any steady state $\xxi^*$ of the polyexponential system \eqref{eq:2d-stoich-polyexp} is  complex balanced. This is explained by \cref{thm:DefZero}, as the graph $G$ that generates \eqref{eq:2d-stoich-polyexp} is weakly reversible and has deficiency zero (see \cref{ex:2d-stoich-polyexp} and \cref{fig:2d-stoich-a}). Furthermore, by \cref{thm:polyexp-HJthm}, we see that for each compatibility class $S_{\xxi_0}$, the system does have a unique complex balanced steady state that is globally stable when restricted to its compatibility class $S_{\xxi_0}$, with Lyapunov function given by \eqref{eq:polyexp-Lyapunov}. Moreover, any two steady states $\xxi^*$ and $\vv\zeta$ satisfy $\vv\zeta - \vv\xi^* \in  S^\perp$. \Cref{fig:2d-stoich-c} shows the vector field on two compatibility classes.
\end{ex}

\section{Complex balanced generalized Lotka--Volterra systems}
\label{sec:CB-GLV}

In this section we leverage the results established for polyexponential systems to derive analogous results for \GLV\ systems, which are the primary focus of our study. We introduced polyexponential systems because they arose from \GLV\ systems under a change of variables. Therefore, the dynamical qualities of polyexponential systems from the last section are also properties of the corresponding \GLV\ systems.

First, recall that if the stoichiometric subspace of an E-graph $G$ is not the whole $\rr^n$, then the polyexponential system is invariant on affine subsets (called compatibility classes) that are obtained by shifting the stoichiometric subspace $S$. The \GLV\ system $\dfrac{d\xx}{dt} = \diag(\xx)\vv f_{(G,\vv\kk)}(\xx)$ generated by the same E-graph, which is diffeomorphic to $\dfrac{d\xxi}{dt} = \vv \phi_{(G,\vv\kk)}(\xxi)$, therefore also has invariant subsets.

\begin{lem}\label{lem:compatibility_manifold}
    Consider the \GLV\ system \eqref{eq:GLV} generated by $(G,\vv\kk)$, with stoichiometric subspace $S$. For any $\xx_0 \in \rrpp^n$, its \df{compatibility manifold} $\exp(S_{\xxi_0}) \coloneqq \exp(S + \xxi_0)$ is invariant, where $\xxi_0 = \log \xx_0$. 
    In particular, for any non-zero $\vv w \in S^\perp$, the equation  $\vv w^\top \log \xx(t) = \vv w^\top  \log \xx(0)$ defines a first integral of \eqref{eq:GLV}.
\end{lem}

\begin{ex}\label{ex:compatibility-glv}
    Recall that the \GLV\ system given by \eqref{eq:2d-stoich-glv} in \Cref{ex:2d-stoich-glv} has a two-dimensional stoichiometric subspace. 
    The compatibility manifold of a point $\xx_0 \in \rrpp^n$ for \eqref{eq:2d-stoich-glv} is the set $\exp(S + \xxi_0)$ where $\xxi_0 = \log(\xx_0)$, as shown in \cref{fig:2d-stoich-b}. Note that any such set is an invariant set for \eqref{eq:2d-stoich-glv}.
\end{ex}

We now define the notion of complex balance for \GLV\ systems.

\begin{defn}\label{def:GLV-CB}
    The \GLV\ system \eqref{eq:GLV} generated by $(G,\vv\kk)$ is said to be \df{complex balanced} if there exists $\xx^* \in \rrpp^n$ such that for any $\yy_i \in V$, we have
    \begin{equation}\label{eq:glv-cb}
        \sum_{j \colon (i,j)\in E} \kk_{ij} (\xx^*)^{\yy_i} = \sum_{j \colon (j,i)\in E} \kk_{ji} (\xx^*)^{\yy_j}.
    \end{equation}
    The point $\xx^*$ is called a \df{complex balanced steady state}.
\end{defn}

\begin{lem}\label{lem:cov-CB}
    Consider the \GLV\ system \eqref{eq:GLV} and the polyexponential system \eqref{eq:polyexponential}, each generated by $(G,\vv\kk)$. Then $\xx^*$ is complex balanced for the \GLV\ system if and only if $\log \xx^*$ is complex balanced for the polyexponential system. 
\end{lem}
\begin{proof}
    The definitions of complex balanced in polyexponential and \GLV\ systems can be rewritten as $e^{\mm Y^\top \xxi} \in \ker \Ak $ and $\xx^{\mm Y} \in \ker \Ak $ respectively. The claim follows from $\xx^{\mm Y} = e^{\mm Y^\top \log \xx}$.
\end{proof}

\begin{ex}\label{ex:glv-cb}
    Consider the \GLV\ system \eqref{eq:2d-stoich-glv} in \cref{ex:2d-stoich-glv}.
    The point $\xx^* \in \rrpp^3$ is complex balanced if it  satisfies the following equations:
    \eq{
    \kk_{12}x_1^* + \kk_{13}x_1^* &= \kk_{21}x_2^* + \kk_{31}x_3^*\\
    \kk_{21}x_2^* + \kk_{23}x_2^* &= \kk_{12}x_1^* + \kk_{32}x_3^*\\
    \kk_{31}x_3^* + \kk_{32}x_3^* &= \kk_{13}x_1^* + \kk_{23}x_2^*.
    }
    Similar to the case of \cref{ex:2d-stoich-polyexp-cb}, all steady states $\xx^*\in \rrpp^3$ are complex balanced according to \cref{lem:cov-CB}. Moreover, we also obtain analogous conclusions as in \cref{ex:2d-stoich-polyexp-Def0}, i.e., for each compatibility manifold $\exp(S_{\xxi_0})$ the system  \eqref{eq:2d-stoich-glv} has a unique complex balanced steady state that is globally stable when restricted to its compatibility manifold $\exp(S_{\xxi_0})$, and any two positive steady states $\xx^*, \vv z^*$ satisfy $\log \vv z - \log \xx^* \in S^{\perp}$. 
\end{ex}

Using \Cref{lem:cov-CB}, we obtain the three theorems below, which are analogues of \Cref{thm:polyexp-HJthm,thm:WRD0isCBforallk,thm:DefZero}.

\begin{thm}\label{thm:glv-cb}
    Consider the \GLV\ system \eqref{eq:GLV} on $\rrpp^n$ generated by  $(G,\vv\kk)$, with stoichiometric subspace $S$. Suppose $\xx^*$ is a complex balanced steady state. Then the following hold. 
    \begin{enumerate}
        \item A positive state $\vv z$ is a steady state if and only if we have $\log \vv z - \log \xx^* \in S^{\perp}$.

        \item Every positive steady state is complex balanced. 
        
        \item There is exactly one positive steady state within each compatibility manifold. 
        
        \item The function 
        \eqn{\label{eq:GLV-Lyapunov}
            V(\xx) = \sum_{i=1}^n (\log x_i - \log x_i^*)^2, 
        }
        defined on $\rrpp^n$, is a global proper Lyapunov function for \eqref{eq:GLV}. When restricted to a compatibility manifold, the function $V(\xx)$ has a unique minimum at the complex balanced steady state contained in that compatibility manifold. In particular, every positive steady state is globally attracting within its compatibility manifold. 
    \end{enumerate}
\end{thm}

The level sets of $V(\xx)$ are images of the sphere centered at $\log \xx^*$ under the component-wise exponentiation map; see \Cref{fig:GLV-Lyapunov} for examples in $\rrpp^2$ and $\rrpp^3$.

\begin{figure}[h!t]
\centering
    \begin{subfigure}[b]{0.4\textwidth}
    \centering
        \includegraphics[height=4cm]{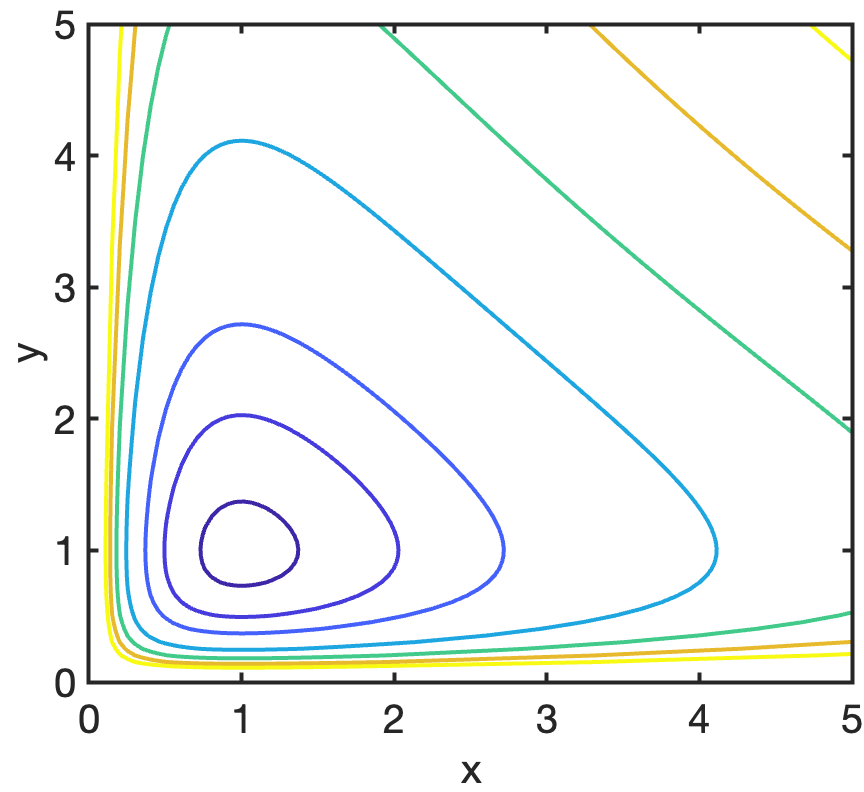}
    \caption{} 
    \label{fig:GLV-Lyapunov-2d}
    \end{subfigure}
    \begin{subfigure}[b]{0.4\textwidth}
    \centering
        \includegraphics[height=4cm]{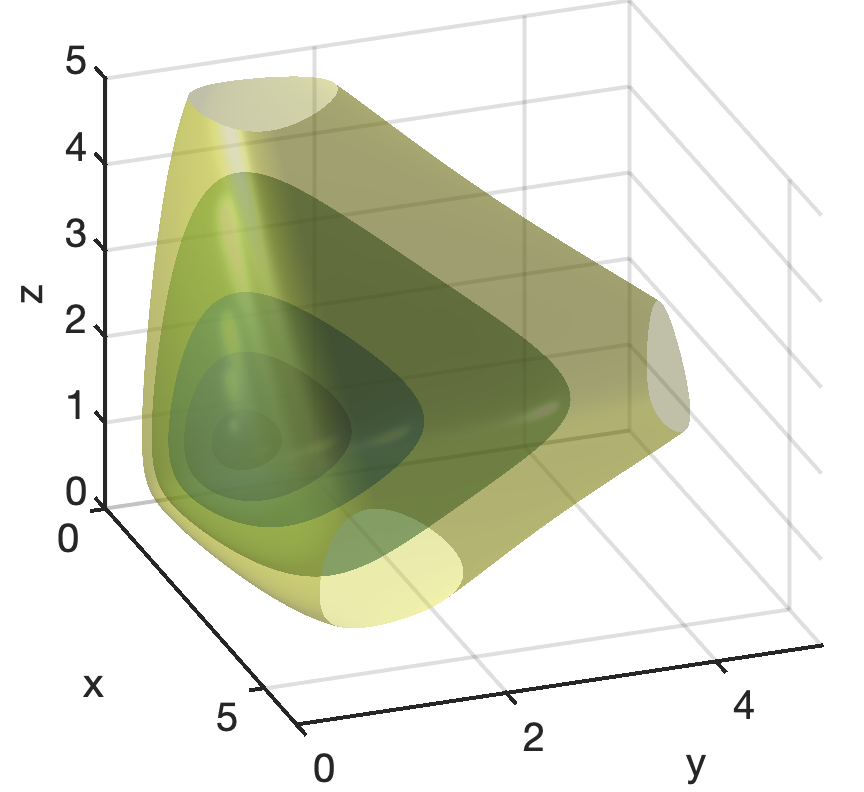}
    \caption{} 
    \label{fig:GLV-Lyapunov-3d}
    \end{subfigure}
\caption{Level sets of the Lyapunov function \eqref{eq:GLV-Lyapunov} in (a) $\rrpp^2$ and (b) $\rrpp^3$, with  $\xx^* = \mathbbm{1}$.}
\label{fig:GLV-Lyapunov}
\end{figure}

\begin{thm}
    The \GLV\ system generated by $(G,\vv\kk)$ is complex balanced for any $\vv \kk > \vv 0$ if and only if $G$ is weakly reversible and has deficiency zero. 
\end{thm}

\begin{thm}\label{thm:glv-DefZero}
    Suppose the E-graph $G$ has deficiency zero, and consider some arbitrary $\vv\kk > \vv 0$.
    \begin{enumerate}
        \item If $G$ is weakly reversible, then within each compatibility manifold the \GLV\ system generated by $(G,\vv\kk)$ has a unique positive steady state that is complex balanced and globally stable.
    
        \item If $G$ is not weakly reversible, then the \GLV\ system generated by $(G,\vv\kk)$ cannot have any positive steady state, nor can it admit any periodic orbit.
    \end{enumerate}
    In particular, if the \GLV\ system \eqref{eq:GLV} has  a positive steady state, then this steady state is complex balanced and globally stable within its compatibility manifold.
\end{thm}

\bigskip

We now extend our main result to a larger class of \GLV\ systems. In \Cref{sec:examples-applications}, we will see in specific examples that this generalization provides a powerful method for obtaining conditions for global stability in applications. In particular, it allows us to recover a classical result that gives conditions for global stability for quadratic cooperative Lotka--Volterra systems (\cref{thm:coop-glv}).

\begin{thm}\label{thm:diag-cb-glv}
    Let $G$ be an E-graph with stoichiometric subspace $S$, and let $\vv\kk > \vv 0$ be a choice of edge weights. Let $\mm D = \diag(\vv d)$ be a $n \times n$ positive diagonal matrix. Consider the system 
    \eqn{\label{eq:diagD-glv}
    \dfrac{d\xx}{dt} = \mm D \diag (\xx) \vv f_{(G,\vv\kk)}(\xx).
    }
    Suppose the \GLV\ system $\dfrac{d\xx}{dt} = \diag (\xx) \vv f_{(G,\vv\kk)}(\xx)$ generated by $(G,\vv\kk)$ has a complex balanced steady state $\xx^*$. Then the following hold:
    \begin{enumerate}
        \item Any positive steady state $\vv z$ of \eqref{eq:diagD-glv} satisfies $\log \vv z - \log \xx^* \in S^{\perp}$. 

        \item For any $\xxi_0 \in \rr^n$, the set $\exp(\mm D S + \xxi_0)$ is an invariant set  of  \eqref{eq:diagD-glv}, and contains exactly one positive steady state.  

        \item The function 
            \eqn{\label{eq:scaled-Lyapunov-glv}
                V(\xx) = \sum_{i=1}^n \frac{1}{d_i} (\log x_i - \log x_i^*)^2 , 
            } 
        defined on $\rrpp^n$ is a global proper Lyapunov function for \eqref{eq:diagD-glv}. In particular, every positive steady state is globally stable within its invariant set.
    \end{enumerate}
\end{thm}

\begin{proof}
    \begin{enumerate}
        \item Under the  change of variables $x_i \mapsto \xi_i = \log x_i$,  the system \eqref{eq:diagD-glv}  is transformed into
        \eqn{\label{eq:diag-polyexp}
        \dfrac{d\xxi}{dt} = \mm D \vv \phi_{(G,\kk)}(\xxi),
        }
        where $\dfrac{d\xxi}{dt} =  \vv \phi_{(G,\kk)}(\xxi)$ is the polyexponential system generated by $(G,\vv\kk)$. Since the set of steady states for \eqref{eq:diag-polyexp} is $\xxi^* + S^\perp$ independent of $\mm D$, the claim (1) follows by reversing back to the $x_i$ variables.
    
        \item The compatibility class of $\xxi_0\in\rr^n$ for \eqref{eq:diag-polyexp} is not $\xxi_0 + S$, but instead $\xxi_0 + \mm D S$. The intersection $(\xxi^* + S^\perp) \cap (\xxi_0 + \mm DS)$ contains exactly one point for any $\xxi^*$; therefore each invariant set $\xxi_0 + \mm D S$ contains exactly one steady state for \eqref{eq:diag-polyexp}. 
        Accordingly, each invariant set $\exp(\mm D S + \xxi_0)$ contains exactly one positive steady state for \eqref{eq:diag-polyexp}. 

        \item Consider the function $L(\xxi) = \sum_i \frac{1}{d_i}(\xi_i - \xi_i^*)^2$, obtained by applying the change of variables to \eqref{eq:scaled-Lyapunov-glv}. Its gradient is $\nabla L(\xxi) = 2 \mm D^{-1} (\xxi - \xxi^*)$. 
        Thus, $\braket{\dfrac{d\xxi}{dt}}{\nabla L(\xxi)} = 2\braket{\vv \phi_{(G,\kk)}(\xxi)}{\xxi-\xxi^*}$, which we have shown in \Cref{lem:polyexp-Lyapunov} to be non-positive, and it is equal to zero if and only if $\xxi - \xxi^* \in S^\perp$. Therefore $L(\xxi)$ is a proper Lyapunov function for \eqref{eq:diag-polyexp}, and each steady state $\vv\zeta$ is globally stable within its compatibility class $\vv \zeta + \mm D S$. By reversing the change of variables we complete the proof of (3). \qedhere 
    \end{enumerate}
\end{proof}

\section{Examples}
\label{sec:examples-applications}

In this section we look at some specific examples, in order to illustrate  how the theory we developed can be used to prove  global stability for various types of \GLV\ systems. We start by using \Cref{thm:glv-DefZero} and \Cref{thm:diag-cb-glv} to recover a classical result. Then, we illustrate how to use \Cref{thm:diag-cb-glv} to find conditions for global stability for a \GLV\ system with higher-order interaction terms.
Finally, we analyze a specific example, i.e., with fixed parameters, and show that it is globally stable using \Cref{thm:glv-cb}.

\subsection{Analysis of a classical global stability result}

The following result is an illustration of 
how to use \Cref{thm:glv-DefZero} and \Cref{thm:diag-cb-glv} to recover a classical result on global stability of quadratic cooperative Lotka--Volterra systems~\cite{Smith1986,Goh1977,Goh1979,Hofbauer_Sigmund_1998}.
    
    \begin{figure}[h!t]
    \centering 
        \begin{tikzpicture}[scale=1]
        \begin{scope}[shift={(5,-3)}]
        \begin{axis}[
          view={115}{25},  
          axis lines=center,
          axis line style={gray, thick},
          width=5.5cm,height=5cm,
          ticks = none, 
          xmin=0,xmax=1.7,ymin=0,ymax=1.6,zmin=0,zmax=1.5,
        ]
        \addplot3 [no marks, dashed, gray!70] coordinates {(1,0,0)  (1,0,2.5)};
        \addplot3 [no marks, dashed, gray!70] coordinates {(0,1,0)  (0,1,2.5)};
        \addplot3 [no marks, dashed, gray!70] coordinates {(2.5,0,1)  (0,0,1) (0,2.5,1)};
        \addplot3 [no marks, dashed, gray!70] coordinates {(1,0,0)  (1,2.5,0) };
        \addplot3 [no marks, dashed, gray!70] coordinates {(0,1,0)  (2.5,1,0) };
        
        \addplot3 [only marks, blue] coordinates {(0,0,0) (0,0,1) (0,1,0) (1,0,0)};
        
        \node [outer sep=1pt] (3) at (axis cs:0,0,1) {};
        \node [outer sep=1pt] (2) at (axis cs:0,1,0) {};
        \node [outer sep=1pt] (1) at (axis cs:1,0,0) {};
        \node [outer sep=1pt] (0) at (axis cs:0,0,0) {};
        
        \begin{scope}[transform canvas={xshift=0.ex, yshift=0.2ex}]
        \draw [fwdrxn, transform canvas={xshift=0.3ex,yshift=0.25ex}] (3)--(2); 
        \draw [fwdrxn, transform canvas={xshift=-0.3ex,yshift=-0.25ex}] (2)--(3); 
        \end{scope}
        \begin{scope}[transform canvas={xshift=0.2ex, yshift=-0.5ex}]
        \draw [fwdrxn, transform canvas={xshift=-0.ex,yshift=0.35ex}] (1)--(2); 
        \draw [fwdrxn, transform canvas={xshift=-0.05ex,yshift=-0.35ex}] (2)--(1); 
        \end{scope}
        \begin{scope}[transform canvas={xshift=-0.4ex, yshift=0.2ex}]
        \draw [fwdrxn, transform canvas={xshift=0.3ex,yshift=-0.25ex}] (3)--(1); 
        \draw [fwdrxn, transform canvas={xshift=-0.3ex,yshift=0.25ex}] (1)--(3); 
        \end{scope}

        \begin{scope}[transform canvas={xshift=0.5ex, yshift=0.35ex}]
        \draw [fwdrxn, transform canvas={xshift=-0.35ex,yshift=0.15ex}] (0)--(1); 
        \draw [fwdrxn, transform canvas={xshift=0.35ex,yshift=-0.15ex}] (1)--(0); 
        \end{scope}
        \begin{scope}[transform canvas={xshift=-0.25ex, yshift=0.1ex}]
        \draw [fwdrxn, transform canvas={xshift=-0.2ex,yshift=-0.35ex}] (0)--(2); 
        \draw [fwdrxn, transform canvas={xshift=-0ex,yshift=0.35ex}] (2)--(0); 
        \end{scope}
        
        \draw [fwdrxn, transform canvas={xshift=2pt,yshift=-2pt}] (0)--(3); 
        \draw [fwdrxn, transform canvas={xshift=-2pt,yshift=-2pt}] (3)--(0); 
    
        \end{axis} 
        \end{scope}
        \end{tikzpicture}
    \caption{E-graph $G$ that generates the \GLV\ system \eqref{eq:coop-glv-graph} for $n=3$. With appropriate choice of edge weights, the system \eqref{eq:diag-coop-glv} is generated by a subgraph of $G$. }
    \label{fig:coop-LV}
    \end{figure}
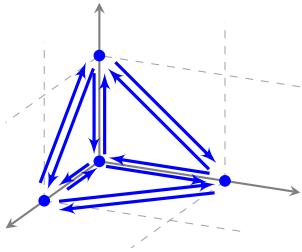 

\begin{thm}[\cite{Goh1977,Hofbauer_Sigmund_1998}]\label{thm:coop-glv}
    Consider the cooperative \GLV\ system 
    \eqn{\label{eq:coop-glv}
    \dfrac{d\xx}{dt} = \textbf{diag}(\xx) (\vv r + \mm A \xx),
    }
    where $r_i > 0$, $[\mm A]_{ij} = a_{ij} \geq 0$ for $i\neq j$, and $a_{ii} < 0$. Suppose that ${\vv x^*} \in \rrpp^n$ is a positive steady state of \eqref{eq:coop-glv}. If there exists $d_1,\ldots,d_n>0$ such that
    \eqn{\label{eq:diag-dominant}
    \sum_{i=1}^n d_i a_{ij} \leq 0
    }
    for all $j$, then \eqref{eq:coop-glv} is globally stable  on $\rrpp^n$, with  ${\vv x^*}$ as its globally attracting steady state. 
\end{thm}
\begin{proof}
    By \cref{thm:diag-cb-glv}, it suffices to find a positive matrix $\mm D^{-1} = \diag(d_1,\ldots, d_n)$ such that
    \eqn{\label{eq:diag-coop-glv}
    \dfrac{d\xx}{dt} 
    &= \textbf{diag}(\xx) (\mm D^{-1} \vv r + \mm D^{-1} \mm A \xx) 
    =  \textbf{diag}(\xx) \left(  
        \sum_{i=1}^n d_i r_i \ehat_i + \sum_{i=1}^n d_i a_{ii} x_i \ehat_i + \sum^n_{\substack{i,j=1 \\i \neq j}} d_j a_{ji} x_i \ehat_j 
     \right)
    }
    can be generated by an E-graph $G^*$ with $\xx^*$ as a complex balanced steady state. Denote by $\ehat_0$ the origin of $\rr^n$, and consider the E-graph $G = (V,E)$ where 
    \eq{
    V &=  \{\ehat_i : 0\leq i\leq n\} 
    \quad \text{and} \quad 
    E = \{\ehat_i \xrightleftharpoons[\kk_{ji}]{\kk_{ij}}  \ehat_j: 0\leq i <  j \leq n\},
    }
    with edge weights $\kk_{ij}$ as indicated (see \Cref{fig:coop-LV} for the three-dimensional version). 
    The \GLV\ system generated by  $(G, \vv\kk)$ is 
    \eqn{\label{eq:coop-glv-graph}
    \dfrac{d\xx}{dt} 
    &=\diag(\xx)\left( 
        \sum_{i=1}^n \kk_{0i} \ehat_i 
        - \sum_{i=1}^n \kk_{i0} x_i \ehat_i 
        + \sum^n_{\substack{i,j=1\\ j \neq i}}  \kk_{ij}x_i (\ehat_j - \ehat_i)
     \right) .
    }
    We will find $\kk_{ij} \geq 0$ such that \eqref{eq:coop-glv-graph} is equal to \eqref{eq:diag-coop-glv}. 
    If any $\kk_{ij} = 0$, we exclude the corresponding edge in the target E-graph $G^*$.
    
    To find $\kk_{ij}$, we set the coefficients for each monomial in \eqref{eq:diag-coop-glv} and \eqref{eq:coop-glv-graph} to be equal. From the constant terms, we obtain $\kk_{0i} = d_i r_i$ for each $i=1,\ldots, n$. From the coefficients of $x_i$, 
    \eq{ 
        -\left( \kk_{i0} + \sum_{j\neq i} \kk_{ij} \right) \ehat_i + \sum^n_{ j \neq i} \kk_{ij} \ehat_j 
        = \sum_{j=1}^n d_j a_{ji} \ehat_j .
    }
    Therefore, for $1 \leq i \neq j \leq n$, we have
    \eq{
    \kk_{0i} = d_i r_i, \quad
    \kk_{ij} = d_j a_{ji}, \quad
    \kk_{i0} = - \sum_{j=1}^n d_j a_{ji},
    }
    where $\kk_{ij} \geq 0$.
    Note that the hypotheses on $\vv r$ and $\mm A$ imply that $\kk_{0i} > 0$,  $\kk_{ij} \geq 0$, and $\kk_{i0} \geq 0$. 

    Therefore, under the stated assumptions, \eqref{eq:coop-glv-graph} is the GLV system generated by $(G^*, \vv\kk)$ where $G^*$ is a subgraph of $G$. 

    To conclude global stability of $\xx^*$ for $\dfrac{d\xx}{dt} = \diag(\xx) (\vv r + \mm A \xx)$ using \cref{thm:diag-cb-glv}, it suffices to show that $\xx^*$ is a complex balanced steady state for $(G^*,\vv\kk)$ and that its stoichiometric subspace is $\rr^n$. 
    Indeed, because $\kk_{0i}>0$, the edge $\vv 0 \to \ehat_i$ is in $G^*$; therefore its stoichiometric subspace is $\rr^n$, and $G^*$ contains $n+1$ vertices, which all lie in the same connected component. Thus, the deficiency of $G^*$ is $\delta = (n+1) - 1 - n = 0$. By \cref{thm:glv-DefZero} we conclude that the positive steady state $\xx^*$  is complex balanced, and thus, by \cref{thm:diag-cb-glv}, globally attracting. 
\end{proof}

\subsection{Analysis of models with higher-order interactions}

We now analyze the global stability of a \GLV\ model with higher-order interaction (HOI) terms~\cite{aladwani2019, singh2020, gibbs2022, billick1994, levine2017}. HOI terms have been introduced to better understand species coexistence in diverse communities. Recent works have shown that higher-order terms can stabilize dynamics~\cite{grilli2017}, shape the diversity of species~\cite{bairey2016}, and capture otherwise unexplained dynamics of ecological systems~\cite{mayfield2017}. In general, the Lotka--Volterra model for three-way HOI~\cite{gibbs2022, singh2020} is
\eq{
\dfrac{dx_i}{dt} = x_i \left(r_i + \sum_{j}a_{ij}x_j + \sum_{j}\sum_{k} b_{ijk} x_j x_k \right).
}
As an example, we will show how to use \cref{thm:diag-cb-glv} in order to analyze global stability of a two-species cooperative Lotka--Volterra system. 
This approach can be extended for higher-dimensional models~\cite{CraciunRojasLaLuz2025}.

\begin{ex}\label{ex:HOI-ex}
For simplicity, consider the following \GLV\ model with HOI terms involving two species~\cite{aladwani2019}:
    \eqn{\label{eq:HOI-ex}
    \begin{split}
    \dfrac{dx_1}{dt} &= x_{1}(r_1 - a_{11} x_1 + a_{12}x_2  - b_1x_1 x_2),\\
    \dfrac{dx_2}{dt} &= x_{2}(r_2 + a_{21} x_1 - a_{22}x_2  - b_2x_1 x_2),
    \end{split}
    }
where $r_i>0$, $a_{ij}\geq 0$, and $b_i\geq 0$ for $i,j=1,2$. Suppose $\xx^* = (x_1^*,x_2^*)$ is a positive steady state\footnote{
    That a positive steady state $\xx^*$ exists follows from the main theorem in \cite{Boros2019} for example.
} 
for the system \eqref{eq:HOI-ex}; we show that $\xx^*$ is a globally attracting point by showing that $\xx^*$ is complex balanced for some \GLV\ system and applying \cref{thm:diag-cb-glv}.

\begin{figure}[h!t]
\centering 
    \begin{tikzpicture}[scale=2]
        \draw [step=1, gray!50!white, thin] (0,0) grid (1.25,1.25);
        \node at (0,1.35) {};
            \draw [->, gray] (0,0)--(1.25,0);
            \draw [->, gray] (0,0)--(0,1.25);
            \node [inner sep=2pt] (1) at (0,0) {\blue{$\bullet$}};
            \node [inner sep=2pt] (2) at (1,0) {\blue{$\bullet$}};
            \node [inner sep=2pt] (3) at (1,1) {\blue{$\bullet$}};
            \node [inner sep=2pt] (4) at (0,1) {\blue{$\bullet$}};

            \node at (0.5,0) [below] {\blue{$\kk_{21}$}};
            \node at (0,0.5) [left] {\blue{$\kk_{41}$}};
            \node at (0.5,1) [below] {\blue{$\kk_{43}$}};
            \node at (1,0.5) [left] {\blue{$\kk_{23}$}};

            \node at (0.5,0) [above] {\blue{$\kk_{12}$}};
            \node at (0,0.5) [right] {\blue{$\kk_{14}$}};
            \node at (0.5,1) [above] {\blue{$\kk_{34}$}};
            \node at (1,0.5) [right] {\blue{$\kk_{32}$}};
            
            \draw [-{stealth}, thick, blue, transform canvas={xshift=0ex, yshift=0.5ex}] (1)--(2) ;
            \draw [-{stealth}, thick, blue, transform canvas={xshift=0ex, yshift=-0.5ex}] (2)--(1) ;
            \draw [-{stealth}, thick, blue, transform canvas={xshift=0.5ex, yshift=0ex}] (1)--(4) ;
            \draw [-{stealth}, thick, blue, transform canvas={xshift=-0.5ex, yshift=-0ex}] (4)--(1) ;

            \draw [-{stealth}, thick, blue, transform canvas={xshift=0.5ex, yshift=0ex}] (3)--(2) ;
            \draw [-{stealth}, thick, blue, transform canvas={xshift=-0.5ex, yshift=-0ex}] (2)--(3) ;
            \draw [-{stealth}, thick, blue, transform canvas={xshift=0ex, yshift=0.5ex}] (3)--(4) ;
            \draw [-{stealth}, thick, blue, transform canvas={xshift=-0ex, yshift=-0.5ex}] (4)--(3) ;
    \end{tikzpicture}
    \caption{E-graph $G$ that generates the \GLV\ system \eqref{eq:HOI-2d-general-square}. With appropriate choice of edge weights, \eqref{eq:HOI-ex-diag} is generated by a subgraph of $G$.}
    \label{fig:HOI-2D-general-square}
\end{figure}
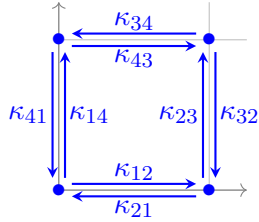 

Consider\footnote{%
While our goal is to analyze the system \eqref{eq:HOI-ex}, we turn our attention to \eqref{eq:HOI-ex-diag}, because this system has additional free parameters $d_1, d_2$ that we can take advantage of.
} instead the rescaled system
    \eqn{\label{eq:HOI-ex-diag}
    \begin{split}
    \dfrac{dx_1}{dt} &= x_{1}d_1(r_1 - a_{11} x_1 + a_{12}x_2  - b_1x_1 x_2),\\
    \dfrac{dx_2}{dt} &= x_{2}d_2(r_2 + a_{21} x_1 - a_{22}x_2  - b_2x_1 x_2),
    \end{split}
    }
for some $d_1, d_2>0$. 
We will find $d_i$ such that \eqref{eq:HOI-ex-diag} is generated by a subgraph of the E-graph $G$ shown in \Cref{fig:HOI-2D-general-square}.  Formally, $G$ and the indicated edge weights generate the \GLV\ system
\eqn{\label{eq:HOI-2d-general-square}
\dfrac{d}{dt}\vc{x_1\\ x_2} &= \diag(x_1,x_2)\left(\vc{\kk_{12}\\ \kk_{14}} + \vc{-\kk_{21}\\ \kk_{23}}x_1 + \vc{\kk_{43}\\ -\kk_{41}}x_2 + \vc{-\kk_{34}\\ -\kk_{32}}x_1x_2\right),
}
where $\kk_{ij}\geq 0$. 
For \eqref{eq:HOI-ex-diag} and \eqref{eq:HOI-2d-general-square} to be equal on $\rrpp^2$, we must have 
\begin{gather*}
    \kk_{12}=d_1r_1,\qquad \kk_{21}=d_1a_{11},\qquad  \kk_{43}=d_1a_{12},\qquad  \kk_{34}=d_1b_1,\\
    \kk_{14}=d_2r_2,\qquad \kk_{23}=d_2a_{21},\qquad \kk_{41}=d_2a_{22},\qquad \kk_{32}=d_2b_2.
\end{gather*}
In addition, we want $\xx^*$ to be complex balanced for \eqref{eq:HOI-2d-general-square}. With the appropriate substitutions for $\kk_{ij}$ and using the fact that $\xx^*$ is a steady state, the complex balance equations \eqref{eq:glv-cb} reduce to 
\eq{
d_1r_1 + d_2r_2 = d_1a_{11}x^*_1+d_2a_{22}x^*_2,
}
which has a solution if and only if
\eqn{\label{eq:sufficient-cond}
\text{sign}(r_1 - a_{11}x_1^*) = \text{sign}(a_{22}x_2^* - r_2).
}
Therefore \eqref{eq:sufficient-cond} is a sufficient condition\footnote{This assumption can be removed and this analysis can be done in a much more general setting, by using a complete graph instead of the graph $G$ in \Cref{fig:HOI-2D-general-square}. This will be addressed in future work~\cite{CraciunRojasLaLuz2025}.} for $\xx^*$ to be a globally attracting point for the system \eqref{eq:HOI-ex}.  
\end{ex}

\bigskip

Our approach for the proof of \cref{thm:coop-glv} and the analysis of \cref{ex:HOI-ex} rely on the fact that if we want to conclude  global stability for a \GLV\ system, it is enough to find an E-graph $G$ generating the system, such that it is complex balanced. The same approach can be used in many other cases. In the following example we  prove global stability of a particular instance of the system \eqref{eq:intro-ex} from the introduction. In future work~\cite{CraciunRojasLaLuz2025} we will describe how to apply this approach systematically in order to  obtain general conditions for \eqref{eq:intro-ex} to be globally stable, using the methods described here.

\begin{ex}\label{ex:generic-ode}
    Consider the dynamical system
    \eqn{
    \begin{split}\label{eq:generic-ode}
        \dfrac{dx_1}{dt} &= x_1(10 x_1^{-2} - 7 - 6x_1 + 4x_2 - x_1^2)\\
        \dfrac{dx_2}{dt} &= x_2(6x_2^{-1} + 5 + 2x_1 - x_2 - 6x_1x_2 - 3x_2^{3/2}).
    \end{split}
    }
    Note that $\xx^* = (1,1) \in\rrpp^2$ is a steady state for this system. We would like to conclude that $\xx^*$ is a globally attracting point on $\rrpp^2$. If we can show that the system \eqref{eq:generic-ode} can be generated by an E-graph $G^*$ and that $\xx^*$ is a complex balanced steady state for $G^*$, then \Cref{thm:glv-cb} implies that $\xx^*$ is a globally attracting point for \eqref{eq:generic-ode}.

    \begin{figure}[h!t]
    \centering
    \begin{tikzpicture}[scale=2]
    \draw [step=1, gray!50!white, thin] (-2,-1) grid (2.25,1.75);
    \node at (0,1.35) {};
    \draw [->, black] (-2,0)--(2.25,0);
    \draw [->, black] (0,-1)--(0,1.75);
    \node [inner sep=2pt] (1) at (0,0) {\blue{$\bullet$}};
    \node [inner sep=2pt] (2) at (1,0) {\blue{$\bullet$}};
    \node [inner sep=2pt] (3) at (0,1) {\blue{$\bullet$}};
    \node [inner sep=2pt] (8) at (1,1) {\blue{$\bullet$}};
    
    \node [inner sep=2pt] (4) at (-2,0) {\blue{$\bullet$}};
    \node [inner sep=2pt] (5) at (0,-1) {\blue{$\bullet$}};
    \node [inner sep=2pt] (6) at (2,0) {\blue{$\bullet$}};
    \node [inner sep=2pt] (7) at (0,1.5) {\blue{$\bullet$}};

    \node at (0.5,0) [below] {\blue{$\kk_{21}$}};
    \node at (0,0.5) [left] {\blue{$\kk_{41}$}};
    \node at (1,0.5) [left] {\blue{$\kk_{32}$}};

    \node at (0.5,0) [above] {\blue{$\kk_{12}$}};
    \node at (0,0.5) [right] {\blue{$\kk_{14}$}};
    \node at (0.5,1) [above] {\blue{$\kk_{43}$}};
    \node at (1,0.5) [right] {\blue{$\kk_{23}$}};

    \node at (-1,0) [below] {\blue{$\kk_{51}$}};
    \node at (-1,0) [above] {\blue{$\kk_{15}$}};

    \node at (0,-0.5) [left] {\blue{$\kk_{61}$}};
    \node at (0,-0.5) [right] {\blue{$\kk_{16}$}};

    \node at (1.5,0) [below] {\blue{$\kk_{72}$}};
    \node at (1.5,0) [above] {\blue{$\kk_{27}$}};

    \node at (0,1.15) [above left] {\blue{$\kk_{84}$\,}};
    \node at (0,1.15) [above right] {\blue{\,$\kk_{48}$}};
    
    \draw [-{stealth}, thick, blue, transform canvas={xshift=-0ex, yshift=0.5ex}] (1)--(2) ;
    \draw [-{stealth}, thick, blue, transform canvas={xshift=0ex, yshift=-0.5ex}] (2)--(1) ;
    \draw [-{stealth}, thick, blue, transform canvas={xshift=0.5ex, yshift=0ex}] (1)--(3) ;
    \draw [-{stealth}, thick, blue, transform canvas={xshift=-0.5ex, yshift=-0ex}] (3)--(1) ;
    \draw [-{stealth}, thick, blue, transform canvas={xshift=0.5ex, yshift=0ex}] (2)--(8) ;
    \draw [-{stealth}, thick, blue, transform canvas={xshift=-0.5ex, yshift=-0ex}] (8)--(2) ;
    \draw [-{stealth}, thick, blue, transform canvas={xshift=-0ex, yshift=0ex}] (3)--(8) ;

    \draw [-{stealth}, thick, blue, transform canvas={yshift=-0.5ex}] (4)--(1) ;
    \draw [-{stealth}, thick, blue, transform canvas={yshift=0.5ex}] (1)--(4) ;
    \draw [-{stealth}, thick, blue, transform canvas={xshift=-0.5ex}] (5)--(1) ;
    \draw [-{stealth}, thick, blue, transform canvas={xshift=0.5ex}] (1)--(5) ;

    \draw [-{stealth}, thick, blue, transform canvas={yshift=-0.5ex}] (6)--(2) ;
    \draw [-{stealth}, thick, blue, transform canvas={yshift=0.5ex}] (2)--(6) ;
    \draw [-{stealth}, thick, blue, transform canvas={xshift=-0.5ex}] (7)--(3) ;
    \draw [-{stealth}, thick, blue, transform canvas={xshift=0.5ex}] (3)--(7) ;

    \end{tikzpicture}
    \caption{E-graph $G$ that generates  \eqref{eq:generic-ode} with the appropriate choice of edge weights.}
    \label{fig:generic-ode-graph}
    \end{figure}
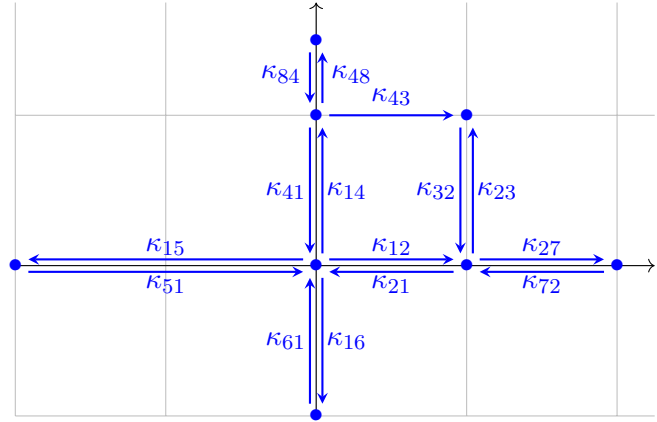 
    
    Consider the \GLV\ system generated by the E-graph $G$ shown in \Cref{fig:generic-ode-graph}:
    \eqn{\label{eq:generic-ode-graph}
    \begin{split}
        \dfrac{d}{dt}\vc{x_1\\ x_2} &= \diag(x_1,x_2)\left(\vc{2\kk_{51} \\ 0}x_1^{-2} + \vc{0\\ \kk_{61}}x_2^{-1} +  \vc{-2\kk_{15} + \kk_{12}\\ -\kk_{16} +\kk_{14}} + \vc{-\kk_{21} + \kk_{27}\\ \kk_{23}}x_1\right. \\
        &\left. + \vc{\kk_{43}\\ -\kk_{41} + \frac{1}{2}\kk_{48}}x_2 + \vc{-\kk_{34}\\ -\kk_{32}}x_1x_2 + \vc{-\kk_{72}\\0}x_1^2 + \vc{0\\-\frac{1}{2}\kk_{84}}x_2^{3/2} \right).
    \end{split}
    }
    Note that \eqref{eq:generic-ode} and \eqref{eq:generic-ode-graph} are equal if we make the following choices for the edge weights in \Cref{fig:generic-ode-graph}:
    \begin{gather*}
        \kk_{12}=3,\qquad  \kk_{21}=7,  \qquad \kk_{14}=11,\qquad  \kk_{41}=7,\qquad  \kk_{23}=2, \qquad \kk_{32}=6,\qquad  \kk_{43}=4, \\ 
        \kk_{34}=0, \qquad
        \kk_{15}=\kk_{51}=5, \qquad \kk_{16}=\kk_{61}=6, \qquad \kk_{27}=\kk_{72}=1, \qquad \kk_{48}=\kk_{84}=6.
    \end{gather*}
    Furthermore, it can be checked directly that $\xx^*=(1,1)$ is complex balanced for $(G,\vv\kk)$. Thus by \Cref{thm:glv-cb}, $\xx^*$ is globally attracting for \eqref{eq:generic-ode}.

    More generally, if we look at the system \eqref{eq:intro-ex} with {\em arbitrary} parameters $I_i, r_i, a_{ij}, b_{ij}$, and we assume that $I_i, a_{ij}, b_{ij}$ are {\em positive}, then there exist positive choices of parameters $\kappa_{ij}$ for the graph $G$ in \Cref{fig:generic-ode-graph} that will give rise to the system \eqref{eq:intro-ex}; this is a key step for obtaining the complex balance property. 
    We will analyze this in depth in future work~\cite{CraciunRojasLaLuz2025}.
\end{ex}

\section{Discussion}
\label{sec:discussion}

In this paper we analyze the global stability of generalized Lotka--Volterra (\GLV) systems of the form~\eqref{eq:GLV_intro}, without restrictions on the number of variables or the degree of the polynomials on the right-hand side. Our main result is \cref{thm:glv-cb}, which says that if the system~\eqref{eq:GLV_intro} has a  steady state that is complex balanced with respect to some graph $G$, then all its steady states are complex balanced, and there exists a foliation of the state space $\rrpp^n$  into compatibility manifolds, such that each compatibility manifold is invariant and contains exactly one complex balanced steady state $\xx^*$, which is globally attracting within its compatibility manifold.  

Our inspiration for the statement and the proof of \cref{thm:glv-cb} comes from the theory of mass-action systems, where the notion of complex balance was first introduced~\cite{HornJackson1972}. Interestingly, \cref{thm:glv-cb} provides a \GLV\ analogue of the {\em global attractor conjecture for mass-action systems}, which has only been proved under additional assumptions on the graph $G$~\cite{Anderson2011, Craciun2013, Gopalkrishnan2014}.

The central idea guiding our application of the theory to specific \GLV\ systems is to uncover {\em an underlying graph structure} that can give rise to that system (referred to as an {\em E-graph}, and inspired by {\em reaction networks} from mass-action kinetics). A well-chosen graph structure (more specifically, finding a graph that is {\em weakly reversible}) allows us to study the global stability of the corresponding dynamical system. This approach  unifies the analysis of different classes of dynamical systems, and  provides a framework for deriving new insights into their behavior. 

An important technical detail in this work is the use of the equivalence between \GLV\ systems and {\em polyexponential systems}; in future work we will exploit this equivalence further, in order to obtain  a proof of the {\em persistence conjecture} for weakly reversible \GLV\ systems~\cite{CraciunRojasLaLuzJin2025}, and in our analysis of {\em disguised complex balance}~\cite{CraciunRojasLaLuz2025}, to significantly relax the complex balance assumption while still being able to conclude global stability.


\section*{Acknowledgments}

This work was supported in part by the National Science Foundation grant DMS-2051568.

\bibliographystyle{siam}
\bibliography{cit}

\end{document}